\documentclass[12pt,english]{smfart}
\usepackage{amscd,amsfonts,amsmath,amssymb,amstext,amsthm}
\usepackage{mathpazo}
\usepackage{hyperref}

\newtheorem{theorem}{Theorem}[section]
\newtheorem{lemma}[theorem]{Lemma}
\newtheorem{proposition}[theorem]{Proposition}
\newtheorem{corollary}[theorem]{Corollary}

\theoremstyle{remark}

\renewcommand{\theequation}{\arabic{section}.\arabic{equation}}

\newcommand{\hq } {{/\kern -.185em/}}
\newcommand{\he } {{\kern -.050em\sim \kern -.050em }}
\newcommand{\eq } {{\kern -.100em\sim \kern -.100em}}
\newcommand{\eqs } {{\kern -.100em\sim }}
\newcommand{\aheq} {{\kern -.100em\sim }}

\title[Hyperbolicity of cycle spaces, automorphism groups of
flag domains]{Hyperbolicity of cycle spaces\\
and\\
Automorphism groups of flag domains} 
\thanks{The author's research for this article was partially supported by
SFB/TR 12 and Schwerpunkt \emph{Representation Theory} of the DFG}
\author{Alan Huckleberry} 
\address{Fakult\"at f\"ur Mathematik, Ruhr-Universit\"at Bochum,\\ 
Universit\"atsstra\ss e 150, 
D-44801 Bochum, Germany, and \newline
School of science and engineering, Jacobs University,
Bremen Compus Ring 1,\newline D-28759 Bremen, Germany}
\email{ahuck@cplx.ruhr-uni-bochum.de}
\date{\today}
\begin{document}
%
\begin {abstract}
\noindent
If $G_0$ is a real form of a complex semisimple Lie group $G$
and $Z$ is compact $G$-homogeneous projective algebraic manifold, 
then $G_0$ has only finitely many orbits on $Z$. Complex analytic properties
of open $G_0$-orbits $D$ (flag domains) are studied.
Schubert incidence-geometry is used to prove the Kobayashi 
hyperbolicity of certain cycle space components $\mathcal {C}_q(D)$.  
Using the hyperbolicity of $\mathcal {C}_q(D)$ and analyzing
the action of $\mathrm {Aut}_{\mathcal O}(D)$ on it, an exact description
of $\mathrm {Aut}_{\mathcal O}(D)$ is given.  It is shown that, except in the 
easily understood case where $D$ is holomorphically 
convex with a nontrivial Remmert reduction, it is a Lie group 
acting smoothly as a group of holomorphic transformations on
$D$. With very few exceptions it is just $G_0$. 
\end {abstract}
\maketitle
\section {Introduction}
Throughout this paper $Z$ denotes a compact projective algebraic manifold
which is homogeneous with respect to a holomorphic action
of a connected complex semi-simple group $G$. Such manifolds
are alternatively described by $Z=G/Q$ where $Q$ is a
(complex) parabolic subgroup of $G$. A Lie subgroup 
$G_0$ of $G$ is said to be a \emph{real form} of $G$ whenever
its Lie algebra $\mathfrak {g}_0$ is defined as the fixed point
space of an antilinear involution $\tau :\mathfrak g\to \mathfrak g$.
Since $G_0$ has only finitely many orbits in $Z$ (\cite {W}), it
has at least one and usually many open orbits $D$.  Manifolds of
the type $Z$ are called \emph{flag manifolds} and the open
$G_0$-orbits $D$ are called \emph{flag domains}.  In all applications
standard decompositions allow us to reduce to the case where
$G_0$ is simple and therefore we assume this.  Note that if
$G_0$ is complex, being embedded in $G$ by an antiholomorphic
map, then $G=G_0\times G_0$.  This is the only situation where
it may not be assumed that $G$ is simple.

It is possible that
$D=Z$. This occurs when $G_0$ is compact and there are
several exceptional noncompact cases as well (\cite {O1,O2},
see Proposition 5.2.1 in \cite {FHW}). From the point of
view of the present work, these cases are not of interest. 
Thus we not only assume that $G_0$ is a simple noncompact real 
form of $G$ but also that $D$ is not compact.

Let us fix a maximal compact subgroup $K_0$ of $G_0$ and denote
by $K$ its complexification in $G$. A basic first example of
Matsuki-duality states that $K_0$ has exactly one orbit $C_0$ in
$D$ which is a complex submanifold (\cite {W}).  It is in fact
the only $K$-orbit in $Z$ which is contained in $D$. 

Fix $q:=\mathrm {dim}_{\mathbb C}(C_0)$ and recall that a 
$q$-cycle in $Z$ is a formal linear combination
$$
C=n_1X_1+\ldots +n_mX_m
$$
where the coefficients $n_i$ are positive integers and
the $X_i$ are compact, irreducible, $q$-dimensional subvarieties
of $Z$.  The set of all such cycles has the natural structure
of complex space (\cite {B}, see \cite {FHW} $\S7.4$ for the
properties needed here).  In the particular case where $Z$
is projective, it is the Chow variety. Its components
are themselves projective algebraic and the $G$-action on
these components is algebraic.  In many situations
we regard $C_0$ as the cycle $1.C_0$ in the cycle space and refer
to it as the \emph{base cycle}.  Since the cycle space is
smooth at $C_0$ (\cite {FHW} Theorem 18.6.1), the notion of the
irreducible component of the cycle space of $Z$ which contains $C_0$
makes sense.  This is denoted by $\mathcal {C}_q(Z)$ and
$$
\mathcal {C}_q(D):=\{C\in \mathcal {C}_q(Z); C\subset D\}\,.
$$  
It should be remarked that, although it is known that $\mathcal C_q(D)$
is smooth at the base point $C_0$, it is an open question as to whether
or not it is everywhere smooth.  

The following is the first main result of this paper.
\begin {theorem}\label {hyperbolicity}
The complex space $\mathcal {C}_q(D)$ is Kobayashi-hyperbolic.
\end {theorem}
The proof of this result is established using the 
\emph{Schubert incidence geometry} defined by special Borel
subgroups of $G$. For this, recall that a Borel subgroup
$B$ has only finitely many orbits $\mathcal O$ in $Z$.
Each is abstractly algebraically isomorphic to some
$\mathbb C^{m(\mathcal O)}$. These are referred to as
\emph{Schubert cells}.  Since the $B$-action on $Z$ is
algebraic, every such orbit is Zariski open in its 
closure. We denote such a closure by $S$ and write
$S=\mathcal O\dot \cup E$, where $E$ is the union of
the (lower-dimensional) $B$-orbits on the boundary
of $\mathcal O$. The varieties $S$, which in fact often
have singularities along $E$, are called \emph{Schubert varieties}.
Observe that the CW-decomposition of $Z$ which is defined
by the Schubert varieties only has cells in even (real) 
dimensions.  Thus they freely generate the integral homology
$H_*(Z,\mathbb Z)$.

If $G_0=K_0A_0N_0$ is an Iwasawa-decomposition and the
Borel group $B$ contains the solvable subgroup $A_0N_0$,
we refer to it as being an Iwasawa-Borel subgroup.
These subgroups can be viewed in the flag manifold $G/B$
of all Borel subgroups as the points on the (unique) closed 
$G_0$-orbit.  

Now consider the set $\mathcal {S}(B)$ of (n-q)-dimensional
Schubert varieties of an Iwasawa-Borel subgroup $B$.
It follows from Poincar\'e duality that at least some
of the $S$ in $\mathcal {S}(B)$ have nonempty (set-theoretic)
intersection with $C_0$.  The following is the essential 
ingredient for a number of considerations in this context
(see Theorem 9.1.1 in \cite {FHW}).
\begin {theorem}
If $S=\mathcal O\dot \cup E\in \mathcal {S}(B)$ is an 
(n-q)-dimensional Schubert variety  
of an Iwasawa-Borel subgroup of $G$ and $S\cap C_0\not=\emptyset$,
then $S\cap C_0$ is a finite subset $\{z_1,\ldots ,z_m\}$ 
which is contained in $\mathcal O$. At each $z_i$ the intersection
$S\cap C_0$ is transversal. The orbit $A_0N_0.z_i=:\Sigma _i$
is open in $S$ and closed in $D$; in particular 
$S\cap D\subset \mathcal O$. Furthermore, 
$\Sigma _i\cap C_0=\{z_i\}$.
\end {theorem}
If $S$ is as in the above Theorem, then we consider the
\emph{incidence variety} 
$$
D_S:=\{C\in \mathcal {C}_q(Z); C\cap E\not=\emptyset\}\,.
$$
It is important that, for appropriate multiplicities of its 
components, this is a Cartier divisor.
(see \cite {FHW}, $\S7.4$, and the Appendix of \cite {HS}).
Theorem \ref{hyperbolicity} is proved by considering  the
maps given by the linear systems defined by the incidence
divisors $D_S$.  These are pieced together in such a way
that after paying the small price of a finite map, 
$\mathcal C_q(D)$ is realized as an open subset of a
hyperbolic domain in a product of projective spaces.
The hyperbolicity of this latter domain is proved by
applying the classically proven fact that the complement
the union of $2n+1$ hyperplanes in general position in
$\mathbb P_n$ is hyperbolic.  

The method of Schubert incidence geometry has played a
role in much of our work over the last years.  In particular,
in \cite {H1} we used this method to prove the hyperbolicity
of the group-theoretically defined cycle space $\mathcal M_D$
which is the connected component containing the base cycle
of the intersection $G.C_0\cap \mathcal C_q(D)$.
The proof here of hyperbolicity in the case of the full cycle space 
$\mathcal {C}_q(D)$ is substantially more involved.   

It should be remarked that, except for a certain Hermitian cases 
which are well-understood,
$\mathcal {M}_D$ only depends on $G_0$, i.e., it doesn't
vary as the $G$-flag manifold $Z$ and the $G_0$-flag domain $D$
vary (see \cite{FHW}).  On the other hand, $\mathcal {C}_q(D)$
varies (sometimes wildly) as $Z$ and $D$ vary 
(Part IV of \cite{FHW}).  Thus it would seem that the
cycle spaces $\mathcal {C}_q(D)$ might provide a wide range
of interesting Kobayashi-hyperbolic spaces which could 
be useful for holomorphically realizing $G_0$-representations.     

Theorem \ref{hyperbolicity} is one of the essential ingredients
in the proof of the following second main result 
of this paper (see $\S\ref{Lie group}$, in particular
Theorem \ref{lie structure}).
\begin {theorem}\label {finite-dimensional}
Unless $D$ is holomorphically convex with nontrivial Remmert
reduction, the automorphism group $\mathrm {Aut}_{\mathcal O}(D)$ 
is a Lie group acting smoothly on $D$ as a group of holomorphic 
transformations. 
\end {theorem}

This leads to a detailed description
of the connected component $\mathrm {Aut}_{\mathcal O}(D)^\circ$
for every flag domain $D$.  In order to prove Theorem \ref{finite-dimensional}
we characterize holomorphic convexity by a cycle condition and show that
if this condition is not satisfied, then the natural map
$\mathrm {Aut}_{\mathcal O}(D)\to \mathrm{Aut}_{\mathcal O}(\mathcal {C}_q(D))$ 
is finite-fibered and is a local homeomorphism
onto its closed image. The desired results then follow 
from the fact that the automorphism
group of a hyperbolic complex space is a Lie group. 

As is indicated in the statment of Theorem \ref{finite-dimensional}
the exception to the finite-dimensionality occurs in a setting
which is optimal from the point of view of complex analysis,
namely where $D$ is holomorphically convex but not Stein.
This can only occur in the case where $G_0$ is of Hermitian
type and even in that case it is very rare. 
The Remmert reduction $D\to \widehat D$, which in this case
is induced from a fibration $Z\to \widehat Z$ of the ambient flag 
manifolds, is a $G_0$-equivariant homogeneous fibration onto the 
$G_0$-Hermitian symmetric space of noncompact type embedded in 
its compact dual. This situation can be characterized in a 
number of ways (see $\S\ref{cycle connectivity}$). 

In $\S6$ it is shown that unless $Z=G/Q$ and $Q$ is a maximal parabolic
subgroup, or equivalently $b_2(Z)=1$, it follows that 
$\tilde G_0:=\mathrm {Aut}_\mathcal O(D)=G_0$.
Even when $b_2(Z)=1$ there are very few exceptional cases.  A
key point for the classification of these cases is that
if $\tilde G_0$ properly contains $G_0$, then the complexification
$\tilde G$ also acts transitively on $Z$. Therefore we
are a position to apply Onishchik's classification
of that situation, i.e., that where two different 
complex simple groups act transitively on $Z$.
\section {Hyperbolicity of the cycle space}
This section is devoted to the proof of Theorem \ref{hyperbolicity}.
For this we at first consider an Iwasawa-decomposition 
$G_0=K_0A_0N_0$ and fix an Iwasawa-Borel subgroup $B$ which
contains $A_0N_0$.  Let $\mathcal {S}_{C_0}(B)$ be the
set of (n-q)-dimensional $B$-Schubert varieties which have
nonempty intersection with the base cycle $C_0$. For 
$S=\mathcal O\dot \cup E\in \mathcal {S}_{C_0}(B)$ let $D$ be 
a Cartier divisor with support on the incidence variety
defined by $E$ and let $L$ be the associated line bundle
on $\mathcal {C}_q(Z)$ with section $s$ defining $D$.

We may choose $G$ to be simply-connected so that $L$ 
has the unique structure of a $G$-bundle with $s$ being
a $B$-eigenvector in the finite-dimensional $G$-representation 
space $\Gamma (\mathcal {C}_q(Z),L)$.  Since this representation
may not be irreducible, we let $V$ be the irreducible subrepresentation
which contains $s$ as a $B$-highest weight vector and consider
the $G$-equivariant meromorphic map
$$
\varphi :\mathcal {C}_q(Z)\to \mathbb P(V^*)
$$
defined by $\varphi (x):=\{s; s(x)=0\}$.

Observe that since $E$ is contained in the complement of $D$,
the restriction of $\varphi $ to $\mathcal {C}_q(D)$ is base
point free. Furthermore, the image $\varphi (\mathcal {C}_q(D))$
is contained in the complement $\mathbb P(V^*)\setminus H_0$ of
the hyperplane which corresponds to the section $s$.

Note that since $H_0$ is $A_0N_0$-invariant, the family
$\mathcal F:=\{gH_0\}_{g\in G_0}$ is compact. As a consequence
the union $\cup_{H\in \mathcal F}H$ is compact and its open
complement $\Omega $ contains the image $\varphi (\mathcal {C}_q(D))$.
In \cite{H1} (see also \cite {FH,FHW}) we utilize the
fact that the d-dimensional $G$-representation on $V^*$ is
irreducible to prove the
following fact.
\begin {theorem} The family $\mathcal F$ contains $2d-1$ hyperplanes
$H_1,\ldots ,H_{2d-1}$ which are in general position in
$\mathbb P(V^*)$.  In particular, $\Omega $ is Kobayashi-hyperbolic.
\end {theorem}
The technical work for the proof of Theorem \ref{hyperbolicity} 
amounts to carrying out the above construction for each of the Schubert 
varieties in $\mathcal {S}(B)$ and then to show that the restriction 
to $\mathcal {C}_q(D)$ of the product of the resulting 
maps is finite-fibered. For a precise statement 
let $S_1,\ldots ,S_m$ be a list of the Schubert 
varieties in $\mathcal {S}_{C_0}(B)$ and
$V_1,\ldots ,V_m$ be the irreducible $G$-representations
with highest weight vectors $s_1,\ldots ,s_m$ obtained as above.
Let $\varphi_j:\mathcal {C}_q(Z)\to \mathbb P(V^*_j)$ be the
associated meromorphic maps.
\begin {proposition} \label{finite map}
The restriction of the meromorphic map
$$
\psi :=\varphi_1\times \ldots \times \varphi_m:
\mathcal {C}_q(Z)\to \mathbb P(V_1^*)\times \ldots \times P(V_m^*)
$$
to $\mathcal {C}_q(D)$ is holomorphic and finite-fibered.
\end {proposition}
The proof of this Proposition, which goes by assuming
the contrary and deriving a contradiction, requires 
some notational preliminaries.  For this 
let $\widehat {\mathcal {C}_q(Z)}$ be the normalized graph
of $\psi $ and 
$$
\widehat {\psi}: \widehat {\mathcal {C}_q(Z)}\to 
P(V_1^*)\times \ldots \times P(V_m^*)
$$
be the resulting holomorphic map.  

We suppose to the contrary
of the claim in Proposition \ref{finite map} that 
$\widehat Y$ is a compact complex curve 
in some $\widehat \psi $-fiber and that $\widehat Y$ has
nonempty intersection with the (biholomorphic) lift
of $\mathcal {C}_q(D)$ in $\widehat {\mathcal {C}_q(Z)}$.
Under that assumption 
let $Y$ denote the projection of $\widehat Y$ in 
$\mathcal {C}_q(Z)$, $\mathfrak X_Y$ the preimage
of $Y$ in the the universal family $\mathfrak X\to \mathcal {C}_q(Z)$
and let $X$ be the the (q+1)-dimensional image of
$\mathfrak X_Y$ in $Z$.

Now consider the intersection of $X$ with any one of the
Schubert varieties in $\mathcal {S}_{C_0}(B)$ which for simplicity we denote by 
$S=\mathcal O\dot \cup E$. Since $X\cap D\not=\emptyset $
and every such $S$ intersects every cycle (transversally) in $D$ in
only finitely many points, it follows that $S\cap X$
contains a 1-dimensional component which has nonempty intersection
with $D$.  Since $\mathcal O\cong \mathbb C^{m(\mathcal O)}$, 
the intersection with $X\cap E$ is nonempty. 
\begin {lemma}
Every $C\in Y$ with $C\cap E\not=\emptyset$ is contained in
the base point set $B_\psi $ of the meromorphic 
map defined by the Schubert variety $S$.
\end {lemma}
\begin {proof}
Let $s$ be the section of the bundle $L$ which is the
$B$-highest weight vector in the representation space
$V$ defined by the incidence divisor $D_S$. Since $E$ is contained
in the complement of $D$, for those cycles 
$\tilde C\in Y\cap \mathcal {C}_q(D)$ it follows that $s(\tilde C)\not=0$.
But $s(C)=0$. Thus if $C$ were not a base point, it would follow that
$\varphi (C)\not=\varphi (\tilde C)$.  
This is contrary to $\widehat Y$ being contained in a 
$\widehat \varphi $-fiber.
\end {proof}
\noindent
{\it Proof of Proposition \ref{finite map}.}
In order to complete the proof, we will produce a  
Schubert variety $S\in \mathcal S_{C_0}(B)$ so that if $C\in Y$ as 
in the above Lemma, then $C$ is \emph{not} a base point of the associated
meromorphic map.  For this we consider the closure 
$\mathrm{cl}(G.C)$ of the $G$-orbit of $C$ in the full
cycle space $\mathcal {C}_q(Z)$ and let $C_1$ be a $B$-fixed
point in a closed $G$-orbit in $\mathrm{bd}(G.C)$.  Then
$C_1$ is the union of finitely many q-dimensional Schubert 
varieties. Recall that Poincar\'e duality in $Z$ is realized
by the intersection pairing of $B$-Schubert varieties with
Schubert varieties of the opposite Borel subgroup $B^*$. 
In particular, there is an (n-q)-dimensional $B^*$-Schubert
variety $S_1=\mathcal O\dot \cup E$ which has nonempty 
finite intersection with $C_1$ with $C_1\cap E=\emptyset $.

Since convergence at the level of cycles (regarded as points)
is the same thing as Hausdorff convergence, it follows
that if $gC$ is sufficiently near $C_1$ in the cycle space,
then $S_1$ has nonempty finite intersection with $gC$ with
$gC\cap E$ likewise being empty.
Consequently there is a proper algebraic subset $A$ of $G$
with the property that if $g\not\in A$, then $C$ has finite
nonempty intersection with $gS_1$ and $C\cap gE=\emptyset $.

Now $S_1$ is not a $B$-Schubert variety, but if we choose
$B$ carefully, the opposite Borel $B^*$ is also an 
Iwasawa-Borel subgroup which is $K_0$-conjugate to $B$. 
For the sake of completeness let us explain this choice
in some detail.  Let $\sigma $ be the involution 
which defines the maximal compact subalgebra $\mathfrak {g}_u$
and which commutes with the involution $\tau $ which defines
$\mathfrak {g}_0$ so that $\sigma \vert \mathfrak {g}_o=:\theta $
is the Cartan involution.  This defines the Cartan decomposition 
$\mathfrak {g}_0=\mathfrak {k}_0\oplus \mathfrak {p}_0$ with
$\mathfrak {a}_0$ defined to be a maximal Abelian subspace
of $\mathfrak {p}_0$.  Let $\mathfrak {m}_0$ be
the centralizer of $\mathfrak {a}_0$ in 
$\mathfrak {k}_0$ and choose a toral subalbra $\mathfrak {t}_0$
in $\mathfrak {m}_0$ so that 
$\mathfrak {h}_0=\mathfrak {t}_0\oplus \mathfrak {a}_0$ is
a real form of the $\sigma $-invariant Cartan subalgebra
$\mathfrak h=\mathfrak t\oplus \mathfrak a$ of $\mathfrak g$.
Finally define a system of positive roots for
$\mathfrak g$ with respect to $\mathfrak h$ so that
the direct sum $\mathfrak {n}_0^+$ of the positive restricted root spaces
defines our Iwasawa decomposition
$\mathfrak {g}_0=\mathfrak {k}_0\oplus 
\mathfrak {a}_0\oplus \mathfrak {n}^+_0$.
Now let $\mathfrak {b}$ be the Borel subalgebra in the minimal
parabolic subalgebra 
$\mathfrak {p}=\mathfrak {m}\oplus \mathfrak {a}\oplus \mathfrak {n}^+$
which consists of $\mathfrak {h}$ together with the positive
root spaces. This is a semidirect sum
$\mathfrak b=
\mathfrak b_{\mathfrak m}\ltimes (\mathfrak {a}\oplus \mathfrak {n}^+)$
where $\mathfrak {b}_{\mathfrak m}$ is a Borel subalgebra of 
$\mathfrak m$. The opposite Borel algebra $\mathfrak b^*$ which consists of 
$\mathfrak {h}$ together with the negative root spaces is of the form
$\mathfrak b^*=\mathfrak {b}_{\mathfrak m}^*\ltimes 
(\mathfrak {a}\oplus \mathfrak {n}^-)$ and
$\mathfrak {g}_0=\mathfrak {k}_0\oplus \mathfrak {a}_0\oplus 
\mathfrak  {n}_0^-$ is the corresponding Iwasawa-decomposition.

Since $K_0$ acts transitively on the Iwasawa-decompositions
with compact summand $\mathfrak {k}_0$, for some $k_0\in K_0$
it follows that $k_0S_1=S$, where $S\in \mathcal {S}_{C_0}(B)$.
As a result, for $g\not\in Ak_o^{-1}$ it follows that $C\cap gS$
is nonempty and finite, and $C\cap gE=\emptyset $.  But
the section $g(s)$ defined by the incidence divisor of $gE$ is
linearly equivalent to that defined by $S$ and $C\cap gE=\emptyset $
translates to $g(s)(C)\not=0$. Consequently $C$ is not a 
base point of the meromorphic map $\varphi $ defined by $S$.
\qed 
      
We now come to the final steps in the proof of
the hyperbolicity of the cycle space.

\medskip\noindent
{\it Proof of Theorem \ref{hyperbolicity}.} Let
$W$ be the image of holomorphic map 
$\widehat {\psi}:\widehat{\mathcal {C}_q(Z)}\to 
\mathbb P(V_1^*)\times \ldots \times \mathbb P(V_m^*)$
and $\widehat {\psi}_{Stein}:\widehat W\to W$ its Stein factorization.
Note that Proposition \ref{finite map} can be reformulated to
state that $\mathcal{C}_q(D)$ is naturally embedded in
$\widehat W$.  Recall that $\mathbb P(V_j^*)$ contains an
open subset $\Omega _j$ which is the complement of the
union of the hyperplanes in a family $\mathcal F_j$, which
is Kobayashi-hyperbolic and which contains the image
of the restriction of $\varphi _j$ to $\mathcal {C}_q(D)$.  
Thus the intersection $\Omega $ of the product
$\Omega _1\times \ldots \times \Omega _m$ with $W$ has
the same properties: It contains the image of $\mathcal {C}_q(D)$
(now regarded as a subset of $\widehat W$) and is Kobayashi-hyperbolic.
To complete the proof observe that the restriction of
$\widehat {\psi}_{Stein}$ to the preimage $\widehat \Omega $ of 
$\Omega $ is a proper finite-fibered holomorphic map.  Since
$\Omega $ is hyperbolic, it follows that $\widehat \Omega $
is likewise hyperbolic (\cite {K}, Proposition 3.2.11).  Finally,
since $\mathcal {C}_q(D)$ has been realized as an open
subset of $\widehat \Omega $, it is immediate that it
is hyperbolic as well.\qed

\section {Cycle reduction}
Here we discuss two equivariant holomorphic equivalence relations
which are defined by the cycles in $D$. The first has been introduced
in [H2], but our point of view here is different and 
for the sake of completeness we provide the essential
details.
\subsection* {Cycle connectivity}\label {cycle connectivity}
Let us say that two points $p,q\in D$ are \emph{connected by cycles} if
there are finitely many cycles $C_1,\ldots ,C_N\in \mathcal {C}_q(D)$
so that $p\in C_1$, $q\in C_N$ and the chain $C_1\cup \ldots \cup C_N$
is connected.  This notion defines an equivalence relation
where two points are equivalent whenever they are connected by
a chain of cycles. It is $G_0$-equivariant in the sense
that for all $g\in G_0$
$$
p\he q\Leftrightarrow g(p)\he g(q).
$$
Let $p_0$ be a base point with the property that $K_0.p_0=C_0$
is the base cycle. If $p$ is connected to $p_0$
by a chain $C_1\cup \ldots \cup C_N$ to $p_0$ with $p_0\in C_1$
and $p\in C_N$ and $k_0\in K_0$, then $k_0(p)$ is connected
to $p_0$ by the chain $C_0\cup k_0(C_1)\cup \ldots \cup k_0(C_N)$.
Thus the equivalence class $[p_0]$ is $K_0$-invariant. 

Now write $D=G_0/H_0$ where $H_0$ is the isotropy subgroup of
the point $p_0$.  The reduction $D\to \widehat D/\he $ defined
by cycle connectivity is $G_0$-equivariant and therefore it is
a homogeneous fibration $G_0/H_0\to G_0/I_0$. If $I_0=H_0$, in other
words if any two points can be connected by cycles, we say that 
$D$ is \emph{cycle connected}.
\begin {proposition}
If $D$ is not cycle connected, then $I_0=K_0$, i.e., 
the base $\widehat D$ of the cycle reduction 
is the associated $G_0$-Hermitian symmetric space. 
\end {proposition}
\begin {proof}
Above it was shown that $K_0$ is contained in the stabilizer of
the equivalence class $[p_0]$.  Recall that $K_0$ is not
only a maximal compact subgroup of $G_0$ but is in fact
a maximal subgroup.  Thus the stabilizer of $[p_0]$ is
either the full group $G_0$ or is $K_0$.  In the former
case $D$ is cycle connected and in the latter case
$[p_0]=C_0$.
\end {proof} 
Observe that in the case where $D$ is not cycle connected 
there is an open neighborhood $U$ of the identity 
in the complex $G$-isotropy subgroup $Q$ at $p_0$ which 
stabilizes the equivalence class $[p_0]=C_0$.  Since $U$
generates $Q$, it follows that $[p_0]$ is $Q$-invariant. Define
$\widehat Q$ as the stabilizer of $[p_0]$ in $G$ and note
that it contains both $K$ and $Q$
\begin {proposition} If $D$ is not cycle connected, then
the restriction $R$ of the fibration
$Z=G/Q\to G/\widehat Q=\widehat Z$ to the flag domain $D$ is the
quotient $G_0/H_0=D\to \widehat D=G_0/K_0$ defined by
cycle connectivity equivalence in the strong sense
that the fibers of $D\to \widehat D$  are the
same as the fibers of $Z\to \widehat Z$ over points of 
$\widehat D$.
\end {proposition}
\begin {proof}
The restriction $R$ maps $D$ onto a flag domain $D_1$ whose
$G_0$-isotropy at its base point contains $K_0$. Since
$D$ is not cycle connected, the neutral fiber of 
$G/Q=Z\to \widehat Z=G/\widehat Q$
is the cycle $C_0$; in particular, $\widehat Z$ is not
just a point.  Since the maximal subgroup
$K_0$ is the stabilizer of $C_0$ in $G_0$, it follows that
$D_1=\widehat D$.
\end {proof}
Let us summarize the results of this paragraph.  
\begin {corollary}
Under the assumption that $D$ is not a Hermitian symmetric space
of noncompact type, the following are equivalent:
\begin {enumerate}
\item
The flag domain $D$ is not cycle connected with 
reduction $G_0/H_0=D\to \widehat D=G_0/K_0$ defined
by cycle connectivity equivalence.
\item
The flag domain $D$ is holomorphically convex with nontrivial
Remmert reduction. In particular, 
the base $\widehat D$ of its Remmert reduction 
$G_0/H_0=D\to \widehat D=G_0/K_0$ is a Hermitian symmetric 
space $\widehat D$ which
is embedded as a $G_0$-orbit in its compact dual $\widehat Z$
and its neutral fiber is $C_0$.
\end {enumerate}
\end {corollary}
\begin {proof}
If $D$ is holomorphically convex, then the stated properties
concerning its Remmert reduction are known (\cite {W}, see also 
\cite {FHW}, p. 63). Thus, in order to show that 2.) implies 1.)
we must only note that, since its base 
is Stein, the Remmert reduction maps every chain of cycles in
$D$ to a point. Therefore $D$ is not cycle
connected. The direction 1.) implies 2.) follows immediately
from that fact that the fiber of the reduction defined
by cycle connectivity equivalence is the compact variety
$C_0$ and the base is the Stein Hermitian symmetric space.
\end {proof}
\subsection* {The group $\mathbf{\mathrm {Aut}_{\mathcal O}(D)}$ 
in the holomorphicallyconvex case}
One of the purposes of this article is to describe in detail
the automorphism groups of flag domains. The case where
$D$ is not cycle connected is special in that 
$\mathrm {Aut}_{\mathcal O}(D)$ is not a Lie group. Let us
now describe this group.

Assume that $D$ is not cycle connected with Remmert reduction
$D\to \widehat D$. Note that since
the Hermitian symmetric space $\widehat D$ is Stein and 
retractible, the bundle $D\to \widehat D$ is holomorphically
trivial.  Thus $D$ is biholomorphically equivalent to the
product $C_0\times \widehat D$ of the base cycle and a Hermitian
symmetric space of noncompact type.  The latter can be
realized as a bounded domain and therefore its automorphism
group is a (known) Lie group acting smoothly on $D$ as a group 
of holomorphic transformations.
 
Since $\mathrm{Aut}_{\mathcal O}(C_0)$ is a complex Lie group,
we may consider the space 
$\mathrm{Hol}(\widehat D,\mathrm {Aut}_{\mathcal O}(C_0))$
of holomorphic maps of the base of the Remmert reduction 
to this group. Every 
$\varphi \in \mathrm{Hol}(\widehat D,\mathrm {Aut}_{\mathcal O}(C_0))$
defines a holomorphic automorphism of the product
$C_0\times \widehat D$, namely by
$(z,w)\mapsto (\varphi (w)(z),w)$.  Conversely, the projection
$\pi :C_0\times \widehat D\to \widehat D$ is equivariant
with respect to $\mathrm {Aut}_{\mathcal O}(C_0\times \widehat D)$
and its ineffectivity on $\widehat D$ is exactly 
$\mathrm{Hol}(\widehat D,\mathrm {Aut}_{\mathcal O}(C_0))$.
Thus, although it is quite large, $\mathrm{Aut}_{\mathcal O}(D)$
can be concretely described.
\begin {proposition}\label {infinite-dimensional}
If $D$ is not cycle connected with cycle equivalence reduction
given by $D\to \widehat D$, then
$$
\mathrm {Aut}_{\mathcal O}(D)\cong 
\mathrm{Hol}(\widehat D,\mathrm {Aut}_{\mathcal O}(C_0))
\rtimes \mathrm {Aut}_{\mathcal O}(\widehat D)\,.
$$
\end {proposition}
The following remark which concludes this paragraph will
be of use in the next section.
\begin {proposition}
If $Z\to Z_1$ is a $G$-equivariant fibration which induces
a mapping $D\to D_1$ of flag domains and $D$ is cycle connected,
then $D_1$ is cycle connected.
\end {proposition}
\begin {proof}
If $D_1$ is not cycle connected, then it is holomorphically
convex; in particular, $\mathcal O(D_1)\not\cong \mathbb C$.
Thus $\mathcal O(D)\not\cong \mathbb C$ and as a result
$D$ is likewise holomorphically convex (\cite {W}).
\end {proof}
\noindent
{\bf Remark.} It should be noted that if we had defined the cycle connectivity
equivalence relation using only the cycles in $\mathcal M_D$,
i.e., those of the form $g(C_0)$ for $g\in G$, then the
entire discussion above could have been repeated. In particular,
$D$ is cycle connected in this sense if and only if it is
cycle connected in the weaker sense where we
use cycles from the full cycle space $\mathcal C_q(D)$.\qed 
\subsection* {Cycle separation}
The \emph{cycle separation} equivalence relation is defined
by $p\he q$ if and only if every cycle in $\mathcal C_q(D)$
which contains $p$ also contains $q$ and vice versa.  In other
words, if $\mathcal F_p$ denotes the family of cycles which contain
$p$, then $p\he q$ if and only if $\mathcal F_p=\mathcal F_q$.
Since $g(\mathcal F_p)=\mathcal F_{g(p)}$ for all 
$g\in \mathrm {Aut}_{\mathcal O}(D)$, 
the equivalence relation is $\mathrm {Aut}_{\mathcal O}(D)$-equivariant. 
In particular, it is $G_0$-equivariant. Therefore, 
if $D=G_0/H_0$, where $H_0$ is the isotropy group of the 
neutral point $p_0$ in $D$, then the reduction $D\to D/\he =:\tilde D$ 
is a $G_0$-fibration $G_0/H_0=D\to \tilde D=G_0/I_0$ where 
$I_0$ is the stabilizer of the equivalence class $[p_0]$.   

For $p\in D$ define 
$\mathcal I_p:=\cap_{C\in \mathcal F_p}C$ and observe that
$$
[p_0]=\cap_{p\in {\mathcal I}_{p_0}}{\mathcal I}_p\,.
$$
The following is an immediate consequence of this description.
\begin {proposition}
The cycle separation equivalence classes $[p]$ are 
closed compact complex analytic subvarieties of $D$.
In particular, if the base cycle $C_0$ is defined by
$K_0.p_0$, then $[p_0]$ is contained in $C_0$.
\end {proposition}
Using this proposition we now show that reduction by
cycle separation can be extended to a $G$-equivariant
fibration $Z\to \tilde Z$ of the ambient projective space.
\begin {lemma}
The equivalence class $[p_0]$ is invariant with respect
to the complex $G$-isotropy group $G_{p_0}=Q$.
\end {lemma}
\begin {proof}
If $U$ is a sufficiently 
small open neighborhood of the identity in 
$Q$, then $U.[p_0]\subset D$. If
$u(p)\not\in [p_0]$ for some $u\in U$ and some $p\in [p_0]$,
then there is either a cycle $C_1$ which contains $p_0$ and not
$u(p)$ or a cycle $C_2$ which contains $u(p)$ and not $p_0$.
In the first case, contrary to $p\in [p_0]$ the cycle $u^{-1}(C_1)$
contains $p_0$ but not $p$. In the second case, $C_2$ is not of 
the form $u(C)$ for some cycle containing $p_0$, because 
$p_0\not\in C_2$.  Thus $u^{-1}(C_2)$
is a cycle containing $p$ but not $p_0$ and as above
this is contrary to $p\in [p_0]$. Consequently $u(p)\in [p_0]$
for all $u\in U$ and $p\in [p_0]$ and, since $U$ generates $Q$
as a group, it follows that $[p_0]$ is $Q$-invariant.
\end {proof}
Therefore the stabilizer $\tilde Q$ of $[p_0]$ in $G$ is 
a parabolic subgroup of $G$ which defines a holomorphic fibration
$G/Q=Z\to \tilde Z:=G/\tilde Q$.  Since $\tilde Q\supset I_0$
and $I_0$ acts transitively on $[p_0]$, it follows that
$\tilde Q/Q=[p_0]$.  This information can be summarized
as follows.
\begin {proposition}
The restriction of the holomorphic fibration 
$\pi:G/Q=Z\to \tilde Z=G/\tilde Q$ to $D$ is the
reduction $\pi_0:G_0/H_0=D\to \tilde D=G/I_0$ of $D$
by cycle separation equivalence. For $p\in D$ the
fibers of the two maps agree: $\pi^{-1}(\pi (p))=\pi_0^{-1}(\pi_0(p))$.  
In particular, $\tilde D$ is a flag domain in $\tilde Z$.
\end {proposition}
Let us say that that \emph{cycles separate the points} 
of a flag domain $D$ if $D=\tilde D$.  The following shows
that one actually gains something by reducing to $\tilde D$.
\begin {proposition}\label {cycles separate}
For every flag domain $D$ the cycles separate the points of
the base $\tilde D$ of the cycle separation equivalence 
relation. 
\end {proposition}
The next lemma is the essential ingredient for the proof
of this Proposition. For this we let
$\tilde q:=\mathrm{dim}(\tilde C_0)$ and denote by $\pi^*$ the
the map from the full space of $\tilde q$-cycles in $\tilde Z$ 
to the full space of $q$-cycles in $Z$ which is induced
from the fibration $\pi :Z\to \tilde Z$.
\begin {lemma}
The restriction of $\pi^*$ to the irreducible component
$\mathcal {C}_{\tilde q}(\tilde Z)$ is a biholomorphic
map $\pi^*:\mathcal {C}_{\tilde q}(\tilde Z)\to \mathcal C_q(Z)$.
In particular, 
$\pi_0^*:\mathcal {C}_{\tilde q}(\tilde D)\to \mathcal C_q(D)$
is also biholomorphic.
\end {lemma}
\begin {proof}
The map $\pi^*$ is holomorphic and injective and therefore its
image lies in some irreducible component of the full cycle
space of $Z$. Since $\pi^*(\tilde C_0)=C_0\in \mathcal C_q(Z)$,
its image is in $\mathcal C_q(Z)$.   Now consider the
map $\pi_*$ from supports of cycles in $\mathcal {C}_q(Z)$
to compact analytic subsets in $\tilde Z$. Let us say that
a cycle $C\in \mathcal {C}_q(Z)$ is $\pi $-saturated if its
image is $\tilde q$-dimensional. Clearly the set of
saturated cycles is closed.  To see that it is also
open, suppose $C$ is saturated, but that there is a sequence
$C_n$ of nonsaturated cycles converging to $C$. It follows
that $\pi (C_n)$ converges to $\pi (C)$ in the Hausdorff
topology. But this is impossible, because 
$\mathrm {dim}(\pi(C_n))>\mathrm {dim}(\pi(C))$.  Since $C_0$ is
saturated, it follows that every $C\in \mathcal C_q(Z)$
is saturated and therefore $\pi_*$ maps $\mathcal C_q(Z)$
to an irreducible subvariety of the full space of $\tilde q$-cycles
in $\tilde Z$. Since $\pi_*(C_0)=\tilde C_0$, this subvariety 
is contained in $\mathcal C_{\tilde q}(\tilde Z)$. Therefore
$\pi^*$ is invertible with inverse $\pi_*$.  The biholomorphicity
of $\pi_0^*$ follows immediately.
\end {proof} 
{\it Proof of Proposition \ref{cycles separate}.} Observe
that if $\tilde p\in \tilde D$, then by the Lemma the family 
$\mathcal F_{\tilde p}$ of cycles
containing $\tilde p$ is mapped biholomorphically by 
$\pi_0^*$ to $\mathcal F_p$ for any $p\in \pi^{-1}(\tilde p)$.
If $\tilde p$ and $\tilde q$ are different points in $\tilde D$
and $p$ and $q$ are corresponding points in $D$, then
$\mathcal F_p\not=\mathcal F_q$.  Thus 
$\mathcal F_{\tilde p}\not=\mathcal F_{\tilde q}$ and consequently
$\tilde p$ and $\tilde q$ are separated by some cycle in
$\tilde D$.\qed

Since $\mathcal C_q(D)$ is Kobayashi-hyperbolic with
$\mathrm {Aut}_{\mathcal O}(\mathcal C_q(D))$ being
a Lie group acting smoothly by holomorphic transformations,
the following formal observation is important for our study of 
$\mathrm {Aut}(D)$.
\begin {proposition}\label{injectivity}
If cycles separate points in $D$, then the canonically defined 
homomorphism 
$\mathrm {Aut}_{\mathcal O}(D)\to \mathrm {Aut}_{\mathcal O}(\mathcal C_q(D))$
is injective.
\end {proposition}
\begin {proof}
If $g\in \mathrm {Aut}(D)$ acts as the identity on the cycle
space, then $g(\mathcal F_q)=\mathcal F_q$ and
consequently $g({\mathcal I}_q)={\mathcal I}_q$ for all $q\in D$.
Recall that $[p]$ is the intersection of such sets 
and $[p]=\{p\}$ in the case where cycles separate points.
Thus if cycles separate points and $g$ acts trivially on the cycle space,
it follows that $g$ acts trivially on $D$.
\end {proof}
As a consequence it follows that for an arbitrary flag domain
$D$ the ineffectivity of the $\mathrm {Aut}_{\mathcal O}(D)$-action
on $\mathcal C_q(D)$ is the same as the ineffectivity of
its action on the base of the
reduction $D\to \tilde D$ defined by cycle separation.
In the next section we show that, unless $D$ is
holomorphically convex with nontrivial Remmert reduction, 
the latter ineffectivity is finite.
\section {Ineffectivity of the action on the cycle space}
In this section it will be assumed that the flag domain
at hand is cycle connected, or equivalently that 
$D$ it is not holomorphically convex. Recall that this
is the same as assuming that $\mathcal O(D)\cong \mathbb C$.
As above $D\to \tilde D$ denotes the reduction by cycle
separation.  

Our goal here is to prove that the ineffectivity
$I$ of the $\mathrm {Aut}_{\mathcal O}(D)$-action on
the cycle space $\mathcal C_{\tilde q}(\tilde D)$ is a finite
(perhaps trivial) group.
For this we must slightly refine our considerations for
the reduction by cycle separation.  This is done as follows.

The reduction by cycle separation $D\to \tilde D$ is constructed
using all cycles in $\mathcal C_q(D)$, because for the purposes
of analyzing $\mathrm {Aut}_{\mathcal O}(D)$ we require it to be
$\mathrm {Aut}_{\mathcal O}(D)$-equivariant.  On
the other hand, by doing this we loose control of the geometric
nature of the cycles, because they are not necessarily 
orbits.  Having paid this price, we now reduce $\tilde D$
by cycle seperation using only those cycles from 
$\mathcal M_{\tilde D}$.  Thus we obtain a sequence
$$
D\to \tilde D\to \tilde D_{res}
$$ 
of $G_0$-homogeneous fibrations. The notation $\tilde D_{res}$ 
refers to the
fact that we have \emph{restricted} to the smaller set of cycles
$\mathcal M_{\tilde D}$.  Note that since 
$\mathcal M_{\tilde D}$ is naturally isomorphic to 
$\mathcal M_D$, we my view $D\to \tilde D_{res}$ as
reduction by cycle separation using $\mathcal M_D$.

The methods which are used for proving the basic properties
of $D\to \tilde D$ may be used for $D\to \tilde D_{res}$. For
example, the sequence of homogeneous fibrations above are
saturated restrictions of fibrations
$$
Z\to \tilde Z\to \tilde Z_{res}\,.
$$
In this case, however, we only know that 
$\pi_{res}:D\to \tilde D_{res}$ defines a $G_0$-equivariant
biholomorphic map from $\mathcal M_{D_{res}}$ to $\mathcal M_D$.

Now we come to the main point: Although $D\to \tilde D_{res}$
is not necessarily $\mathrm {Aut}_{\mathcal O}(D)$-equivariant,
it is $I$-invariant!  This is just due to the fact that
$D\to \tilde D_{res}$ can be defined by continuing
the $I$-invariant map $D\to \tilde D$ to 
$D\to \tilde D\to \tilde D_{res}$.  Using this we will show
that $I$ is a (finite-dimensional) Lie group acting smoothly
on $D$ as a group of holomorphic transformations.
\subsection* {The Lie group structure of $I$}
Our work here makes strong use of the proposition below.
It will be applied to the associated fibration 
$\pi :C\to C_{res}$ of
a cycle $C$ in $\mathcal M_D$ to $C_{res}\in \mathcal M_{D_{res}}$
which is defined by the restriction of the reduction
$D\to \tilde D_{res}$.  Note that if $C=g(C_0)$, then
$C$ and $C_{res}$ are homogeneous with respect to 
$\widehat {G}:=gKg^{-1}$ and $\pi $ is $\widehat G$-equivariant.
\begin {proposition} Let 
$\pi: \widehat G/Q_1=Z_1\to Z_2=\widehat G/Q_2$ be a homogeneous
fibration of flag manifolds.  Let $I_1$ be a Lie group
acting effectively on $Z_1$ as a closed subgroup of 
$\mathrm {Aut}_{\mathcal O}(Z_1)$ which stabilizes every $\pi $-fiber
$\pi^{-1}(z_2)=:F$.  Then $I_1$ acts effectively on 
every $\pi $-fiber.
\end {proposition}
\begin {proof}
We may replace $I_1$ by the (algebraic)
subgroup $I$ of the algebraic group $\mathrm {Aut}_{\mathcal O}(X)$ 
which stabilizes the $\pi $-fibers. 
Now $\mathrm {Aut}_{\mathcal O}(X)^\circ $ is semisimple
and $I^\circ$ is a product of its simple factors. Thus $Z_1=A\times B$,
where $A$ is the corresponding product of irreducible factors of $Z_1$.
Consequently  $A$ is contained in every $\pi $-fiber and as a result
$I^{\circ}$ acts effectively on every $\pi $-fiber as well. 
Finally, since $I$ is an algebraic group $I^\circ$ is semisimple, 
any normal subgroup $\Gamma $ of $I$ is finite. Therefore,
if $\Gamma $ is the ineffectivity of the $I$-action on a $\pi $-fiber 
$F$, at any point $z\in F$ we may faithfully linearize the
$\Gamma $-action on the tangent space $T_zZ_1$ to produce
a local submanifold $\Delta $ complemenary to $F$ at $z$
which is $\Gamma $-invariant and where $\Gamma $ acts faithfully. 
This is contrary to $\Gamma $
stabilizing all $\pi $-fibers.
\end {proof}
The following is the main result of this section.
\begin {theorem} \label{ineffectivity}
If $D$ is a flag domain which is not holomorphically
convex, then the ineffectivity $I$ of the $\mathrm {Aut}(D)$-action
on the base of the reduction $D\to \tilde D$ defined by
cycle separation is a Lie group acting smoothly as a group
of holomorphic transformations on $D$.
\end {theorem}
Before giving the proof, let us note that by the basic
result of Bochner and Montgomery (\cite {BM}, Theorem 4)
it is enough to show that $I$ is a Lie group acting 
\emph{continuously} on $D$.  Thus we begin by looking 
a bit carefully at the topological properties of $I$.  For this let 
$\nu: \mathfrak X\to \mathcal M_D$ be the universal
family where $\mathfrak X=\{(z,C)\in D\times \mathcal M_D; z\in C\}$
and $\nu $ is the projection on the second factor.   
Again we emphasize that $\mathrm {Aut}_{\mathcal O}(D)$
may not act on $\mathfrak X$ but $I$ does, because
it stabilizes every cycle in $\mathcal M_D$.  Since $\nu $ is proper
and $I$-invariant, we may 
apply the theorem of Arzela-Ascoli-Montel 
to prove the following fact.
\begin {proposition}
Every sequence $\{g_n\}\in I$ has a subsequence which converges
in $I$.
\end {proposition}
\begin {proof}
The application of the theorem of Arzela-Ascoli-Montel shows that
after passing to a subsequence 
$\{g_n\}$ converges to $g\in \mathrm {Aut}_{\mathcal O}(\mathfrak X)$.
But this implies that $g_n\vert C\to g\vert C$ for every
$C\in \mathcal M_D$.  Now view $g_n$ as a sequence in
$\mathrm {Aut}_{\mathcal O}(D)$. Given $x\in C\subset D$ the image point
$g_n(x)$ doesn't depend on $C$.  Thus the map
$g:D\to D$ is well-defined by $g(x):=g\vert C(x)$.
In other words $g$ descends via the map $\mathfrak X\to D$ 
to an automorphism of $D$.  
\end {proof}
Given $C\in \mathcal {M}_D$ we let $R_C:I\to \mathrm {Aut}_{\mathcal O}(C)$,
$g\mapsto g\vert C$, the restriction map. The following is an 
immediate consequence of the the above compactness
statement and the continuity of $R_C$.
\begin {corollary}
The restriction map $R_C:I\to \mathrm {Aut}_{\mathcal O}(C)$ 
is a closed map.
\end {corollary}
\noindent
{\it Proof of Theorem \ref {ineffectivity}.}
Consider the $I$-invariant fibration $\pi:D\to \tilde D_{res}$ defined
by cycle separation with cycles in $\mathcal M_D$. If 
$C=g(C_0)\in \mathcal M_D$, then the restriction $\pi _C$ 
of $\pi $ to $C$
maps $C$ to a cycle $\tilde C_{res}$ in $\mathcal M_{\tilde D_{res}}$.
Both cycles are homogeneous with respect to 
$\widehat G:=gKg^{-1}$ and $\pi _C:C\to \tilde C_{res}$ is 
$\widehat G$-equivariant.  Let $I_C$ be the image of $I$
in $\mathrm {Aut}_{\mathcal O}(C)$ and $I_C(F)$ the image of $I_C$
in the automorphism group of a given $\pi_C$-fiber $F$.  The
above Propostion shows that the map $I_C\to I_C(F)$ defined
by restriction is an isomorphism.

Since $D$ is not holomorphically convex, we know that it
is cycle connected with respect to cycles in $\mathcal M_D$.
Thus $\tilde D_{res}$ is cycle connected with respect to
cycles in $\mathcal M_{\tilde D_{res}}$. Given two points
$\tilde p$ and $\tilde q$ in $\tilde D_{res}$ we connect
them with a chain $\tilde C^1_{res}\cup \ldots \cup\tilde C^m_{res}$
with intersection points $\tilde p_j$ so that $\tilde p_0=\tilde p$
and $\tilde p_n=\tilde q$.  Let $F_j$ be the $\pi $-fibers over
the $\tilde p_j$. The above argument shows that if $g\in I$
is in the kernel of the restriction map $I\to I_{C^j}$, then
it is in the kernel of the restriction map $I\to I_{C^j}(F_j)$.
Again applying the above argument, this implies that 
$g$ is in the kernel of $I\to I_{C^{j+1}}$. Using the fact that
$D$ is cycle connected, this shows that if
$g$ is in the kernel of the restriction map $I\to I_C(F)$ for
some cycle $C$ and some fiber $F$, then $g=\mathrm {Id}$,
i.e., the homomorphism $R_C :I\to I_C(F)$ is an injective
continuous homomorphism onto a subgroup of the Lie 
group $\mathrm {Aut}_{\mathcal O}(F)$.  It follows from the above Lemma 
that $R_C$ is a homeomorphism onto a closed subgroup $I_C(F)$ of 
$\mathrm {Aut}_{\mathcal O}(F)$.  Since $I_C(F)$ is a Lie group, it
follows follows that $I$ is a Lie group (see Theorem 1.1 of
Chapter VIII in [H]). Since it is a (closed) subgroup 
$\mathrm {Aut}_{\mathcal O}(D)$, the action of $I$ on $D$
is continuous. \qed

We have now shown that ineffectivity $I$ 
is a Lie group. It is normalized by $G_0$ and, since both
$I$ and $G_0$ are subgroups of the topological group
$\mathrm {Aut}_{\mathcal O}(D)$, the action of $G_0$ on
$I$ by conjugation is continuous. By the results in Chapter IX
of \cite{H}, in particular Theorem 3.1, the semidirect
product $I\rtimes G_0$ is a Lie group which is acting 
continuously on $D$. The following is therefore a
consequence of the Theorem of Bochner and Montgomery (\cite {BM}).
\begin {corollary}\label{semidirect}
The Lie group $I\rtimes G_0$ acts smoothly as a group
of holomorphic transformations of $D$.
\end {corollary}

\subsection* {Finiteness of the ineffectivity}
After all this discussion about the ineffectivity $I$ we 
are now in a position to show that it is essentially nonexistent:
Except when $D$ is holomorphically convex with a nontrivial
Remmert reduction, $I$ is finite! The following is the remaining 
tool needed for proving this.
\begin {proposition}\label {action extension}
If $\tilde G_0$ is a connected Lie group acting effectively on $D$
as a group of holomorphic transformations and $\tilde G_0\supset G_0$, 
then the action $\tilde G_0\times D\to D$ extends to a smooth
action $\tilde G_0\times Z\to Z$ by holomorphic transformations.
\end {proposition}
\begin {proof}
Let $n:=\mathrm {dim}_{\mathbb C}(D)$ and consider the space
$\tilde V$ of sections of the anticanonical bundle of $D$
spanned by elements of the form $\xi _1\wedge \ldots \wedge \xi_n$
where the $\xi_j$ are arbitrary holomorphic vector fields defined
by the local $\tilde G$-action on $D$.
Define the $\tilde {\mathfrak g}$-anticanonical map
$\varphi :D\to \mathbb P(\tilde V^*)$ by 
$p\mapsto \{s\in \tilde V; s(p)=0\}$. Since $\tilde G_0$ acts
on $\tilde V$ and $D$ is homogeneous,
$\varphi $ is a $\tilde G_0$-equivariant holomorphic map
(see \cite {HO} for other basic properties of this map).

Let $p_0$ be the neutral point in $D$ and $z_0:=\varphi (p_0)$
its image in $\mathbb P(\tilde V^*)$. Now $D$ is embedded
in the $G$-flag manifold $Z$ whose anticanonical bundle is
very ample.  In fact the $\mathfrak g$-anticanonical map
extends to $Z=G/Q$ as the Tits fibration $G/Q\to G/N$, where
$N$ is the normalizer of $Q$ in $G$.  Since $N=Q$ (Borel's 
Normalizer Theorem), the 
$\mathfrak g$-anticanoncal map of $D$ is an embedding. Since
the vector space $V$ which is defined by limiting
the $\xi _j$ to fields defined by $\mathfrak g$ is
contained in $\tilde V$, it is therefore immediate that
$\varphi $ is an embedding.

Thus we may regard $D$ as 
$\tilde G_0$-orbit $\tilde G_0.z_0$ in $\mathbb P(\tilde V^*)$
and define $\tilde G$ to be the smallest complex Lie subgroup 
containing the group $\tilde G_0$ in $\mathrm {GL}(\tilde V)$.  Hence
$D$ can be regarded as the open $\tilde G_0$-orbit $\tilde G_0.z_0$
in the complex orbit $\tilde G.z_0$.  

After this lengthy background, we consider the complex
orbit $G.z_0$ which is an open subset in $\tilde G.z_0$.
Since $\varphi $ is an embedding, the isotropy algebra
$\mathfrak g_{z_0}$ is just the Lie algebra $\mathfrak q$.
Thus the Lie algebra of $G_{z_0}$ is $\mathfrak q$ and
therefore $G_{z_0}=Q$, in particular $G.z_0=G/Q=Z$ is compact.
Since $G.z_0$ is open in $\tilde G.z_0$ and both groups are
connected, it follows that $\tilde G.z_0=G.z_0=Z$ as desired.
\end {proof}  
\begin {theorem}\label {discrete ineffectivity}
Unless $D$ is holomorphically convex with nontrivial
Remmert reduction, the ineffectivity
$I$ of the action of $\mathrm {Aut}_{\mathcal O}(D)$ on
the cycle space $\mathcal C_q(D)$ is finite.
\end {theorem}
\begin {proof}
Note that the connected component $I^\circ$ is normalized
by $G_0$ and Corollary \ref{semidirect} implies that  
$\tilde G_0:=I^\circ\rtimes G_0$
is a connected Lie group which acts smoothly as a group
of holomorphic transformations on $D$.  The ineffectivity of
this action is contained in the center of $G_0$ and is therefore
discrete.  By definition $G_0$ acts on $Z$ and it follows
from the above Proposition that $\tilde G_0$ acts on $Z$.
By definition the complexification $G$ acts transitively
on $Z$ and therefore the smallest complex subgroup 
$\tilde G$ of $\mathrm {Aut}_{\mathcal O}(Z)$ which contains
$\tilde G_0$ also acts transitively.  

We claim that
$\tilde G$ properly contains $G$. To see this note that
since $\tilde G$ acts transitively on $Z$, it is semisimple.
Thus the Lie algebra $\mathfrak i$ is semisimple 
and $\tilde {\mathfrak g}_0$ is the direct sum 
$\mathfrak {i}\oplus \mathfrak {g}_0$. As a result 
$\tilde {\mathfrak g}$ is not simple.  Thus if 
$\mathfrak g=\tilde {\mathfrak g}$, it follows that
$\mathfrak g$ is not simple and therefore $\mathfrak {g}_0$
is (abstractly) a simple complex Lie algebra embedded
antiholomorphically in its complexification 
$\mathfrak g=\mathfrak {g}_0\oplus \mathfrak {g}_0$. But
if $\mathfrak g=\tilde {\mathfrak g}$, then the image of
$\mathfrak {g}_0$ in $\mathfrak g$ centralizes 
$\mathfrak i$.  On the other hand this image has trivial centralizer.
Thus $\tilde{\mathfrak g}$ properly contains $\mathfrak g$.

As a result we are in the situation where two complex
Lie groups $G$ and $\tilde G$ act transitively on
$Z$. Since $\tilde G$ properly contains $G$, it follows
that $Z=G/Q$ where $Q$ is a \emph{maximal} parabolic 
subgroup (\cite {O1,O2}).  But this situation arises when
$D\to \tilde D$ is a nontrivial reduction by cycle separation.
Since this reduction is the restriction of a fibration
$Z\to \tilde Z$ of the ambient flag manifold, $Q$ is \emph{not}
a maximal parabolic subgroup of $G$.  Thus $\mathfrak i$ is
trivial and therefore $I$ is discrete. But since it is normalized,
and therefore centralized by $G_0$, it is finite.
\end {proof}

\section {Lie group properties}\label{Lie group}
Let us begin by stating the main result of this section.
For this we recall that if a flag domain $D$ is not
cycle connected, then it is holomorphically
convex and the reduction $D\to \widehat D$ by cycle
connectivity equivalence is its Remmert reduction.
The base $\widehat D$ is a Hermitian symmetric space
of noncompact type and is therefore realizable as
a bounded Stein domain.  In particular, if $D=\widehat D$, 
then $\mathrm{Aut}_{\mathcal O}(D)$ is a Lie group acting
smoothly on $D$ as a group of holomorphic transformations.
If $D$ is not cycle connected and $D\not=\widehat D$, then
$\mathrm {Aut}_{\mathcal O}(D)$ is infinite-dimensional, but
nevertheless easily describable (see \ref{infinite-dimensional}).
Here we shall prove that if $D$ is cycle connected, then
$\mathrm {Aut}_{\mathcal O}(D)$ is a Lie group acting smoothly
on $D$ as a group of holomorphic transformations.  This
then completes the proof of the following result.
\begin {theorem}\label{lie structure}
If $D$ is a flag domain with reduction $D\to \widehat D$
by cycle connectivity, then one of the following three
cases holds:
\begin {enumerate}
\item
$D$ is not cycle connected, in which case
it is holomorphically convex. The reduction $D\to \widehat D$
is the Remmert reduction of $D$ has positive-dimensional
fiber $F$ and positive-dimensional base $\widehat D$ 
which is a Hermitian symmetric space of noncompact type. 
The group of automorphisms, which can be described by
$$
\mathrm {Aut}_{\mathcal O}(D)\cong 
\mathrm{Hol}(\widehat {D}, \mathrm {Aut}_{\mathcal O}(F))
\rtimes \mathrm {Aut}_{\mathcal O}(\widehat {D})\,,
$$
is infinite-dimensional.
\item
$D$ is a Hermitian symmetric space of 
noncompact type, i.e., $D=\widehat D$, and $\mathrm{Aut}_{\mathcal O}(D)$ 
is an Lie group acting smoothly on $D$ as a group of holomorphic
transformations.
\item
$D$ is cycle connected, i.e., $\widehat D$ is just a point, and
$\mathrm {Aut}_{\mathcal O}(D)$ is a Lie group acting smoothly
on $D$ as a group of holomorphic transformations.
\end {enumerate}
\end {theorem}
Let us now complete the proof of this theorem.  We
must only handle the case where $D$ is cycle connected or
equivalently where $\mathcal O(D)\cong \mathbb C$.
The proof uses the fact that 
$\mathrm {Aut}_{\mathcal O}(D)$ can essentially be identified
with a subgroup of $\mathrm {Aut}_{\mathcal O}(\mathcal {C}_q(D))$.
The following is a first step in this direction.
\begin {proposition}\label{closed map}
If $D$ is any flag domain, then the canonically defined 
continuous homomorphism
$$
\iota :\mathrm {Aut}_{\mathcal O}(D)\to 
\mathrm {Aut}_{\mathcal O}(\mathcal C_q(D))
$$
is a closed map.
\end {proposition}
\begin {proof}
Let $\mathfrak X\subset D\times \mathcal C_q(D)$ be the 
universal family of cycles with its projections
$\mu:\mathfrak X\to D$ and $\nu:\mathfrak X\to \mathcal C_q(D)$.
Suppose $F$ is a closed subset of $\mathrm {Aut}_{\mathcal O}(D)$
and that $g_n=\iota (h_n)$ defines a sequence in $\iota (F)$ which
converges to $g\in \mathrm {Aut}_{\mathcal O}(\mathcal C_q(D))$.
Observe that the $g_n$ act on $\mathfrak X$ and denote them
there by $\tilde g_n$.  Since $\nu:\mathfrak X\to \mathcal C_q(D)$
is proper and $g_n\to g$, it follows that the sequence 
$\{\tilde g_n\}$ is equibounded.  It therefore follows from
the theorem of Arzela-Ascoli that after going to a
subsequence $\tilde g_{n_k}\to \tilde g$, where $\tilde g$
is a lift of $g$ to $\mathfrak X$.  Now the $\tilde g_n$
are lifts of the $h_n$ and since $\tilde g_n\to \tilde g$,
the automorphism $\tilde g$ of $\mathfrak X$ descends to
an automorphism $h$ of $D$ with $h_n\to h$.  Since $F$
is closed, $h\in F$ and, since $\iota $ is continuous,
$\iota(h)=g$.  Thus $\iota (F)$ is likewise closed.
\end {proof}
\begin {corollary}
If cycles separate points in $D$, then 
$$
\iota: \mathrm {Aut}_{\mathcal O}(D)\to 
\mathrm {Aut}_{\mathcal O}(\mathcal C_q(D))
$$
is an injective homeomorphism onto a closed
subgroup of $\mathrm {Aut}_{\mathcal O}(\mathcal C_q(D))$.
\end {corollary}
\begin {proof}
It remains to prove the injectivity. However, this is
the content of Proposition \ref{injectivity}.
\end {proof}
Using the Theorem of Bochner and Montgomery (\cite {BM}, Theorem 4)
the above corollary shows that Theorem \ref{lie structure} holds
in the case where $D=\tilde D$. In the case where the
fiber of the reduction $D\to \tilde D$ is positive-dimensional
we make use of the following result. 
\begin {proposition}
The continuous homomorphism 
$$
\pi_*:\mathrm {Aut}_{\mathcal O}(D)\to 
\mathrm {Aut}_{\mathcal O}(\tilde D)
$$
which is induced by cycle separation $\pi :D\to \tilde D$
is a closed map.
\end {proposition}
\begin {proof}
Recall that $\pi $ induces a biholomorphic map
$\mathcal C_{\tilde q}(\tilde D)\to \mathcal C_q(D)$.  If
$F$ is a closed subset of $\mathrm {Aut}_{\mathcal O}(D)$
and $g_n=\pi_*(h_n)\in \pi_*(F)$ defines a sequence which converges
to $g\in \mathrm {Aut}_{\mathcal O}(\tilde D)$, then
the associated sequence of automorphisms of 
$\mathcal C_{\tilde q}(\tilde D)$ converges to the automorphism
of the cycle space associated to $g$.  Thus the sequence
$\{\iota (h_n)\}$ on the cycle space $\mathcal C_q(D)$ which 
is associated to $\{h_n\}$ converges and the result follows
from Proposition \ref{closed map}.
\end {proof}
As a result we know that the image of $\pi_*$ is
a closed subgroup of $\mathrm {Aut}_{\mathcal O}(\tilde D)$
and is therefore a Lie subgroup.  Recall that the 
$\mathrm {Ker}(\pi_*)=I$ is finite.  Thus
the map 
$
\pi_*:\mathrm {Aut}_{\mathcal O}(D)\to 
\mathrm {Aut}_{\mathcal O}(\tilde D)
$
is a topological covering map of a Lie group. As a result
$\mathrm {Aut}_{\mathcal O}(D)$ is homeomorphic to a Lie group
and again by the Theorem of Bochner and Montgomery it follows that
$\mathrm {Aut}_{\mathcal O}(D)$ is a Lie group acting smoothly
on $D$ as a group of holomorphic transformations.  This
completes the proof of Theorem \ref{lie structure}.  
\section {Detailed description  of 
$\mathbf{\mathrm {Aut}_{\mathcal O}(D)}$}
Recall that our initial setting is that of a simple
real form $G_0$ of a complex semisimple group $G$ acting
on a $G$-flag manifold $Z=G/Q$.  Except in the case where
$D$ holomorphically convex with nontrivial Remmert reduction,
where $\mathrm {Aut}_{\mathcal O}(D)$ is a certain
precisely described infinite-dimensional group, 
$\mathrm {Aut}_{\mathcal O}(D)$ is a Lie group acting smoothly
on $D$ by holomorphic transformations. In this section
we restrict to the latter case and give a more detailed description
of $\mathrm {Aut}_{\mathcal O}(D)$. Here is a first step in
that direction.
\begin {proposition}
If $Q$ is not maximal, i.e., if the Betti number $b_2(Z)$ is
at least two, then
$\mathrm {Aut}_{\mathcal O}(D)^\circ=G_0$.
\end {proposition}
\begin {proof}
If $\mathrm {Aut}_{\mathcal O}(D)^\circ =:\tilde G_0$, then by 
Proposition \ref{action extension} the action of $\tilde G_0$
extends to a smooth action $\tilde G_0\times Z\to Z$ of
holomorphic transformations on $Z$.  Since $\tilde G_0\supset G_0$,
the smallest complex Lie group $\tilde G$ in $\mathrm {Aut}_{\mathcal O}(Z)$
which contains $\tilde G_0$ also acts transitively on $Z$.
As we observed in the proof of Theorem \ref{discrete ineffectivity}, if
$\tilde G_0$ properly contains $G_0$, then $\tilde G$ properly contains
$G$.  Thus it follows from the work of Onishchik (\cite {O1,O2}) that
$Q$ is maximal.
\end {proof}
\begin {corollary}
If $G$ is defined to be $\mathrm{Aut}_{\mathcal O}(Z)^\circ$ and $D$
is not holomorphically convex with nontrivial Remmert reduction, then
$\mathrm{Aut}_{\mathcal O}(D)^\circ=G_0$.
\end {corollary}
\subsection* {Exceptional cases}
If $G$ is not the connected component $\tilde G$ 
of the full automorphism group $\mathrm{Aut}_{\mathcal O}(Z)$ 
then, as we just noted, $Q$ is a maximal parabolic subgroup
of $G$ and we are in one of the situations classified by
Onishchik (see \cite {O1,O2}).  There are two series of flag
manifolds $Z$ where this is possible and one additional
isolated example:
\begin {itemize}
\item
Odd dimensional
projective spaces where $\tilde G=\mathrm {SL}_{2n}(\mathbb C)$
and $G=\mathrm {Sp}_n(\mathbb C)$.
\item
The space of isotropic
n-planes with respect to the standard complex bilinear form
on $\mathbb C^{2n}$ where $\tilde G=\mathrm {SO}_{2n}(\mathbb C)$
and $G=\mathrm {SO}_{2n-1}(\mathbb C)$.
\item
The 5-dimensional quadric where $\tilde G=\mathrm {SO}_7(\mathbb C)$
and $G=G_2$.
\end {itemize}

\bigskip\noindent
{\bf Example.}  Let $V$ be a complex vector space
and equip $W=V\oplus V^*$ with its standard symplectic
form which is defined by
$$
\omega (v+\varphi ,v'+\varphi ')=
\varphi '(v)-\varphi (v)\,.
$$
Let $\{e_1,\ldots ,e_n\}$ be a basis of $V$.  
Translating to the numerical space, for
$z,w\in \mathbb C^{2n}$ it follows that
$\omega (z,w)=z^tJw$ where
\begin {gather*}
J:=
\begin {pmatrix}
0 & -\mathrm {Id}\\
\mathrm {Id} & 0
\end {pmatrix}
.
\end {gather*}
Define the Hermitian form $h$ on $\mathbb C^{2n}$
by 
$$
h(z,w)=\frac{i}{2}\bar {z}^tJw\,.
$$
It is of signature $(n,n)$. For example a maximal
negative (resp. positive) space is spanned by vectors
of the form $(a,ia)$ (resp. $(a,-ia)$).

Let $\tilde G:=\mathrm {SL}_{\mathbb C}(W)$ and
$G:=\mathrm {Sp}_{\mathbb C}(W,\omega )$. Define
$\tilde G_0$ to be the subgroup of $h$-isometries 
in $\tilde G$ and $G_0=\tilde G_0\cap G$.  Observe that
$\tilde G_0\cong \mathrm{SU}(n,n)$ and 
$G_0\cong \mathrm {Sp}_{2n}(\mathbb R)$.

One directly checks that $\tilde G_0$ has exactly three
orbits in $Z=\mathbb P(W)$, namely the spaces $D^+$ and
$D^-$ of positive (resp. negative) (complex) lines
and the real hypersurface $\Sigma $ of isotropic lines.

To determine the orbit structure of the $G_0$-action
is convenient to use Matsuki duality.  For this
we choose the maximal compact subgroup $K_0$ of $G_0$
to be the copy of the unitary group $U_n$ which acts
diagonally on $W$, i.e., $k(v+\varphi)=k(v)+k(\varphi)$.
The complexification $K$ of $K_0$ in $G$ is just 
$\mathrm {GL}_n(\mathbb C)$. The duality theorem states that 
there is a natural bijective correspondence 
$\mu $ between the $K$-orbits
and the $G_0$-orbits in $Z$. This sends a 
$G_0$-orbit $\mathcal O$ to the $K$-orbit $\mu (\mathcal O)=K.p$ 
of any point $p\in \mathcal O$ with $K_0.p$ \emph{the} minimal 
$K_0$-orbit in $\mathcal O$.  Furthermore, 
$\mu (\mathcal O)\cap \mathcal O$ is just the minimal $K_0$-orbit
$K_0.p$.  

We mention this ``duality'',  because it is often easier to
understand the $K$-orbit structure than it is to understand
the $G_0$-orbit structure.  In this case one checks directly
that $K$ has four orbits in $Z$, namely the closed
orbits $\mathbb P(V)$ and $\mathbb P(V^*)$, a 1-codimesional
orbit with two ends whose closure $Y$ consists of the 
orbit together with the projective subspaces $\mathbb P(V)$
and $\mathbb P(V^*)$, and the complement $Z\setminus Y$ which is the
unique open $K$-orbit. Duality implies that the closed $K$-orbits
are the base cycles in the open $G_0$-orbits. Thus $G_0$ has
two open orbits which are of course contained in the two open
$\tilde G_0$-orbits $D^+$ and $D^-$. The smaller group $G_0$
stabilizes the closed $\tilde G_0$-orbit $\Sigma $. There
the real points $\Sigma _0$ are $G_0$-invariant.  Thus $G_0$
has at least two orbits in $\Sigma $. But altogther it has
only four orbits and therefore we have accounted for all of
them: $D^+$, $D^-$, $\Sigma _0$ and the complement 
$\Sigma \setminus \Sigma _0$.  In particular, $D^+$ and
$D^-$ are exceptional in the sense that the smaller
real form $G_0$ acts transitively. 
\qed

It would be of interest to have a complete list of 
all such exceptional cases, in particular to determine if
and when a given $G_0$-flag domain is \emph{properly} contained
in the corresponding $\tilde G_0$-flag domain. Arguing as above one 
could possibly compile such a list. 
This would, however, seem to require a somewhat 
involved case-by-case discussion 
which would not be appropriate for the present paper.

\subsection* {Acknowledgements}
At the initial stages of this project we profitted greatly 
from discussions with A. Isaev at the Austrailian National
University (ANU) in Canberra. We are thankful
for the funding from the Austrailian Research Council for making
this possible. We are also grateful to the Tata Institute
for Fundamental Research (TIFR) in Mumbai for providing us
with optimal conditions during the latter stages of the project.

\end {document}